\documentclass[12pt, a4paper]{article}%
\usepackage{amsmath,amsthm}
\usepackage{amsfonts}
\usepackage{amssymb}
\usepackage{graphicx}
\usepackage{amsmath}
\usepackage{manfnt}%
\providecommand{\U}[1]{\protect\rule{.1in}{.1in}}

\newtheorem{theorem}{Theorem}[section]

\newtheorem{fact}[theorem]{Fact}

\newtheorem{corollary}[theorem]{Corollary}

\newtheorem{lemma}[theorem]{Lemma}

\newtheorem{problem}[theorem]{Problem}
\newtheorem{proposition}[theorem]{Proposition}
\theoremstyle{definition}
\newtheorem{remark}[theorem]{Remark}

\newtheorem{definition}[theorem]{Definition}

\makeatletter
\let\@fnsymbol\@arabic
\makeatother
\begin{document}

\title{Invariance groups of finite functions and \\orbit equivalence of permutation groups}
\author{
\makebox[0.4 \linewidth]{Eszter~K.~Horv\'{a}th\thanks{Bolyai Institute, University of Szeged, Aradi
v\'{e}rtan\'{u}k tere 1, H-6720 Szeged, Hungary} \footnotemark[3] \footnotemark[4]}
\\{\small \texttt{{horeszt@math.u-szeged.hu}}}
\and 
\makebox[0.4 \linewidth]{G\'{e}za~Makay\footnotemark[1] \footnotemark[3]}
\\{\small \texttt{{makayg@math.u-szeged.hu}}}
\and 
\makebox[0.4 \linewidth]{Reinhard~P\"{o}schel\thanks{Institut f\"{u}r Algebra, Technische
Universit\"{a}t Dresden, D-01062 Dresden, Germany}}
\\{\small \texttt{{Reinhard.Poeschel@tu-dresden.de}}}
\and 
\makebox[0.4 \linewidth]{Tam\'{a}s~Waldhauser\footnotemark[1] \thanks{Partially supported by the
T\'{A}MOP-4.2.1/B-09/1/KONV-2010-0005 program of the National
DevelopmentAgency of Hungary.} \thanks{Partially supported by the Hungarian
National Foundation for Scientific Research under grant no. K83219.}
\thanks{Partially supported by the Hungarian National Foundation for
Scientific Research under grant no. K77409.}}
\\{\small \texttt{{twaldha@math.u-szeged.hu} }} }
\date{}
\maketitle

\begin{abstract}
Which subgroups of the symmetric group $S_{n}$ arise as invariance groups of
$n$-variable functions defined on a $k$-element domain?\ It appears that the
higher the difference $n-k$, the more difficult it is to answer this question. For
$k\geq n$, the answer is easy: all subgroups of $S_{n}$ are invariance groups.
We give a complete answer in the cases $k=n-1$ and $k=n-2$, and we also give a
partial answer in the general case: we describe invariance groups when $n$ is
much larger than $n-k$. The proof utilizes Galois connections and the
corresponding closure operators on $S_{n}$, which turn out to provide a
generalization of orbit equivalence of permutation groups. We also present
some computational results, which show that all primitive groups except for
the alternating groups arise as invariance groups of functions defined on a
three-element domain.

\end{abstract}

\section{Introduction\label{section introduction}}

This paper presents a Galois connection that facilitates the study of
permutation groups representable as invariance groups of functions of several
variables defined on finite domains. We shall assume without loss of
generality that our functions are defined on the set $\mathbf{k}:=\left\{
1,\ldots,k\right\}  $ for some integer $k\geq2$. We say that an $n$-ary
function $f\colon\mathbf{k}^{n}\rightarrow\mathbf{m}$ is \emph{invariant under
a permutation }$\sigma\in S_{n}$, if%
\[
f\left(  x_{1},\dots,x_{n}\right)  =f\left(  x_{1\sigma},\dots,x_{n\sigma
}\right)
\]
holds for all $\left(  x_{1},\ldots,x_{n}\right)  \in\mathbf{k}^{n}$, and we
denote this fact by $\sigma\vdash f$. The \emph{invariance group (or symmetry
group) of }$f$ is the subgroup $\left\{  \sigma\in S_{n}\mid\sigma\vdash
f\right\}  $ of the full symmetric group $S_{n}$. We will say that a group
$G\leq S_{n}$ is $\left(  k,m\right)  $\emph{-representable} if there exists a
function $f\colon\mathbf{k}^{n}\rightarrow\mathbf{m}$ whose invariance group
is $G$. Furthermore, we call a group $\left(  k,\infty\right)  $%
\emph{-representable} if it is $\left(  k,m\right)  $-representable for some
natural number $m$. Note that $\left(  k,\infty\right)  $-representability is
equivalent to being the invariance group of a function $f\colon\mathbf{k}%
^{n}\rightarrow\mathbb{N}$.

A group $G\leq S_{n}$ is $\left(  2,2\right)  $-representable if and only if
it is the invariance group of a Boolean function (i.e., a function
$f\colon\left\{  0,1\right\}  ^{n}\rightarrow\left\{  0,1\right\}  $), and a
group is $\left(  k,\infty\right)  $-representable if and only if it is the
invariance group of a pseudo-Boolean function (i.e., a function $f\colon
\left\{  0,1\right\}  ^{n}\rightarrow\mathbb{R}$, cf.\ \cite[Chapter
13]{CrHa11}). Invariance groups of (pseudo-)Boolean functions are important
objects of study in computer science (see \cite{CloK91} and the references
therein); however, our main motivation comes from the algebraic investigations
of A.~Kisielewicz \cite{Kis1998}. Kisielewicz defines a group $G$ to be
$m$\emph{-representable} if there is a function $f\colon\left\{  0,1\right\}
^{n}\rightarrow\mathbf{m}$ whose invariance group is $G$ (equivalently, $G$ is
$\left(  2,m\right)  $-representable), and $G$ is defined to be
\emph{representable} if it is $m$-representable for some positive integer $m$
(equivalently, $G$ is $\left(  2,\infty\right)  $-representable). It is easy
to see that a group is representable if and only if it is the intersection of
$2$-representable groups (i.e., invariance groups of Boolean functions). It
was stated in \cite{CloK91} that every representable group is $2$%
-representable; however, this is not true: as shown by Kisielewicz
\cite{Kis1998}, the Klein four-group is $3$-representable but not
$2$-representable. Moreover, it is also discussed in \cite{Kis1998} that it is
probably very difficult to find another such example by known constructions
for permutation groups.

In this paper we focus on $\left(  k,\infty\right)  $-representability of
groups for arbitrary $k\geq2$. It is straightforward to verify that a group is
$\left(  k,\infty\right)  $-representable if and only if it is the
intersection of invariance groups of operations $f\colon\mathbf{k}%
^{n}\rightarrow\mathbf{k}$ (cf. Fact~\ref{fact closed vs repr}). We introduce
a Galois connection between operations on $\mathbf{k}$ and permutations on
$\mathbf{n}$, such that the Galois closed subsets of $S_{n}$ are exactly the
groups that are representable in this way. Our main goal is to characterize
the Galois closed groups; as it turns out, the difficulty of the problem
depends on the gap $d:=n-k$ between the number of variables and the size of
the domain. The easiest case is $d\leq0$, where all groups are closed (see
Proposition~\ref{prop k>=n}); for $d=1$ the only non-closed groups are the
alternating groups (see Proposition~\ref{prop k=n-1}). The case $d=2$ is
considerably more difficult (see Proposition~\ref{prop k=n-2}), and the
general case, which includes representability by invariance groups of Boolean
functions, seems to be beyond reach. However, we provide a characterization of
Galois closed groups for arbitrary $d$ provided that $n$ is much larger than
$d$ (more precisely, $n>\max\left(  2^{d},d^{2}+d\right)  $; see
Theorem~\ref{thm main}.)

Let us mention that our approach is also related to \emph{orbit equivalence}
of groups
(see Section~\ref{sec2}(A)). In the case $k=2$, two groups have the same
Galois closure if and only if they are orbit equivalent, whereas the cases
$k>2$ correspond to finer equivalence relations on the set of subgroups of
$S_{n}$. Thus our Galois connection provides a parameterized version of orbit
equivalence that could be interesting from the viewpoint of the theory of
permutation groups.

In Section~\ref{section preliminaries} we formalize the Galois connection, we
discuss its relationship to orbit equivalence, and we recall some basic facts
about subdirect products of groups. We state our main result
(Theorem~\ref{thm main}) in Section~\ref{section general}, where we prove it
in the special cases $d\leq0$ and $d=1$, and we also make some general
observations about closures of direct and subdirect products. We prove
Theorem~\ref{thm main} in Section~\ref{section proof of thm main}, and in
Section~\ref{section computations} we present results of some computer
experiments, which, together with Theorem~\ref{thm main}, settle the case
$d=2$. Finally, in Section~\ref{section concluding} we relate our approach to
relational definability of permutation groups (cf. \cite{Wi69}) and we
formulate some open problems.

\section{Preliminaries\label{section preliminaries}\label{sec2}}


Throughout the paper, $n$ and $k$ denote positive integers; we always assume
that $n,k\geq2$, and we denote the difference $n-k$ by $d$. As usual, $S_{B}$
and $A_{B}$ denote the symmetric and alternating groups, respectively, on an
arbitrary set $B$, and $S_{n}$ stands for the symmetric group on the set
$\mathbf{n}=\{1,\ldots,n\}$.


\begin{center}
\vspace{1.5ex} \textbf{(A) A Galois connection for invariance groups }
\end{center}

In order to precisely state the problem that we study, first we introduce some
terminology and notation. The correspondence $\vdash$ defined in
Section~\ref{section introduction} induces a Galois connection between
permutations of $\mathbf{n}$ and $n$-ary operations on $\mathbf{k}$. More
precisely, let $O_{k}^{(n)}=\left\{  f\mid f\colon\mathbf{k}^{n}%
\rightarrow\mathbf{k}\right\}  $ denote the set of all $n$-ary operations on
$\mathbf{k}$, and for $F\subseteq O_{k}^{(n)}$ and $G\subseteq S_{n}$ let%
\renewcommand{\arraystretch}{1.5}%
\[%
\begin{tabular}
[c]{ll}%
$F^{\vdash}:=\{\sigma\in S_{n}\mid\forall f\in F:\sigma\vdash f\},$ &
$\overline{F}^{\left(  k\right)  }:=(F^{\vdash})^{\vdash},$\\
$G^{\vdash}:=\{f\in O_{k}^{(n)}\mid\forall\sigma\in G:\sigma\vdash f\},$ &
$\overline{G}^{\left(  k\right)  }:=(G^{\vdash})^{\vdash}.$%
\end{tabular}
\
\]

As for every Galois connection, the assignment $G\mapsto\overline{G}^{\left(
k\right)  }$ is a closure operator on $S_{n}$, and it is easy to see that
$\overline{G}^{\left(  k\right)  }$ is a subgroup of $S_{n}$ for every subset
$G\subseteq S_{n}$ (even if $G$ is not a group). For $G\leq S_{n}$, we call
$\overline{G}^{\left(  k\right)  }$ the \emph{Galois closure of }$G$\emph{
over }$\mathbf{k}$, and we say that $G$ is \emph{Galois closed over
}$\mathbf{k}$ if $\overline{G}^{\left(  k\right)  }=G$. Sometimes, when there
is no risk of ambiguity, we will omit the reference to $\mathbf{k}$, and speak
simply about (Galois) closed groups and (Galois) closures. Similarly, we have
a closure operator on $O_{k}^{(n)}$; the study of this closure operator
constitutes a topic of current research of the authors. However, in this paper
we focus on the \textquotedblleft group side\textquotedblright\ of the Galois
connection; more precisely, we address the following problem.

\begin{problem}
\label{problem main}For arbitrary $k,n\geq2$, characterize subgroups of
$S_{n}$ that are Galois closed over $\mathbf{k}$.
\end{problem}

As we shall see, this problem is easy if $k\geq n$, and it is very hard if $n$
is much larger than $k$. Our main result is a solution in the intermediate
case, when $d=n-k>0$ is relatively small compared to $n$. Complementing this
result with a computer search for small values of $n$, we obtain an explicit
description of Galois closed groups for $n=k-1$ and $n=k-2$ for all $n$.
Observe that if $k_{1}\geq k_{2}$, then $\overline{G}^{\left(  k_{1}\right)
}\leq\overline{G}^{\left(  k_{2}\right)  }$, hence if $G$ is Galois closed
over $\mathbf{k}_{2}$, then it is also Galois closed over $\mathbf{k}_{1}$.
Thus we have the most non-closed groups in the Boolean case (i.e., in the case
$k=2$), whereas for $k\geq n$ every subgroup of $S_{n}$ is Galois closed (see
Proposition~\ref{prop k>=n}).

The following fact appears in \cite{CloK91} for $k=2$, and it remains valid
for arbitrary $k$. We omit the proof, as it is a straightforward
generalization of the proof of the equivalence of conditions (1) and (2) in
Theorem~12 of \cite{CloK91}.

\begin{fact}
\label{fact closed vs repr}A group $G\leq S_{n}$ is Galois closed over
$\mathbf{k}$ if and only if $G$ is $\left(  k,\infty\right)  $-representable.
\end{fact}


\begin{center}
\vspace{1.5ex} \textbf{(B) Orbits and closures }
\end{center}

The symmetric group $S_{n}$ acts naturally on $\mathbf{k}^{n}$: for $a=\left(
a_{1},\ldots,a_{n}\right)  \in\mathbf{k}^{n}$ and $\sigma\in S_{n}$, let
$a^{\sigma}=\left(  a_{1\sigma},\ldots,a_{n\sigma}\right)  $ be the action of
$\sigma$ on $a$. We denote the orbit of $a\in\mathbf{k}^{n}$ under the action
of the group $G\leq S_{n}$ by $a^{G}$, and we use the notation
$\operatorname{Orb}^{\left(  k\right)  }\left(  G\right)  $ for the set of
orbits of $G\leq S_{n}$ acting on $\mathbf{k}^{n}$:%
\[
a^{G}:=\left\{  a^{\sigma}\mid\sigma\in G\right\}  ,\qquad\operatorname{Orb}%
^{\left(  k\right)  }\left(  G\right)  :=\left\{  a^{G}\mid a\in\mathbf{k}%
^{n}\right\}  .
\]

Clearly, $\sigma\vdash f$ holds for a given $\sigma\in S_{n}$ and $f\in
O_{k}^{(n)}$ if and only if $f$ is constant on the orbits of (the group
generated by) $\sigma$. Therefore, for any $G,H\leq S_{n}$, we have
$G^{\vdash}=H^{\vdash}$ if and only if $\operatorname{Orb}^{\left(  k\right)
}\left(  G\right)  =\operatorname{Orb}^{\left(  k\right)  }\left(  H\right)
$. On the other hand, from the identity $G^{\vdash\vdash\vdash}=G^{\vdash}$
(which is valid in any Galois connection), it follows that $G^{\vdash
}=H^{\vdash}$ is equivalent to $\overline{G}^{\left(  k\right)  }=\overline
{H}^{\left(  k\right)  }$. Thus we have%
\begin{equation}
\overline{G}^{\left(  k\right)  }=\overline{H}^{\left(  k\right)  }%
\iff\operatorname{Orb}^{\left(  k\right)  }\left(  G\right)
=\operatorname{Orb}^{\left(  k\right)  }\left(  H\right)
\label{eq k-orbit equivalence}%
\end{equation}
for all subgroups $G,H$ of $S_{n}$.

Two groups $G,H\leq S_{n}$ are \emph{orbit equivalent}, if $G$ and $H$ have
the same orbits on the power set of $\mathbf{n}$ (which can be identified
naturally with $\mathbf{2}^{n}$), i.e., if $\operatorname{Orb}^{\left(
2\right)  }\left(  G\right)  =\operatorname{Orb}^{\left(  2\right)  }\left(
H\right)  $ holds \cite{Inglis,SW}. One can define a similar equivalence
relation on the set of subgroups of $S_{n}$ for any $k\geq2$ by
(\ref{eq k-orbit equivalence}), and each class of this equivalence relation
contains a greatest group, which is the common closure of all groups in the
same equivalence class. In other words, a group is Galois closed over
$\mathbf{k}$ if and only if it is the greatest group among those having the
same orbits on $\mathbf{k}^{n}$ (cf. Theorem~2.2 of \cite{Kis1998} in the
Boolean case). Therefore, the Galois closure of $G$ over $\mathbf{k}$ can be
described as follows:%
\begin{equation}
\overline{G}^{\left(  k\right)  }=\left\{  \sigma\in S_{n}\mid\forall
a\in\mathbf{k}^{n}:~a^{\sigma}\in a^{G}\right\}  .
\label{eq closure via orbits}%
\end{equation}

Orbit equivalence of groups has been studied by several authors; let us just
mention here a result of Seress \cite{Ser97} that explicitly describes orbit
equivalence of primitive groups (see \cite{SY} for a more general result). For
the definitions of the linear groups appearing in the theorem, we refer the
reader to \cite{DixMo}.

\begin{theorem}
[\cite{Ser97}]\label{thm seress}If $n\geq11$, then two different primitive
subgroups of $S_{n}$ are orbit equivalent if and only if one of them is
$A_{n}$ and the other one is $S_{n}$. For $n\leq10$, the nontrivial orbit
equivalence classes of primitive subgroups of $S_{n}$ are the following:%
\renewcommand{\theenumi}{(\roman{enumi})}
\renewcommand{\labelenumi}{\upshape\theenumi}%

\begin{enumerate}
\item for $n=3$: $\left\{  A_{3},S_{3}\right\}  ;$

\item for $n=4$: $\left\{  A_{4},S_{4}\right\}  ;$

\item for $n=5$: $\left\{  C_{5},D_{10}\right\}  $ and $\left\{
\operatorname{AGL}\left(  1,5\right)  ,A_{5},S_{5}\right\}  ;$

\item for $n=6$: $\left\{  \operatorname{PGL}\left(  2,5\right)  ,A_{6}%
,S_{6}\right\}  ;$

\item for $n=7$: $\left\{  A_{7},S_{7}\right\}  ;$

\item for $n=8$: $\left\{  \operatorname{AGL}\left(  1,8\right)
,\operatorname{A \Gamma L}\left(  1,8\right)  ,\operatorname{ASL}\left(
3,2\right)  \right\}  $ and $\left\{  A_{8},S_{8}\right\}  ;$

\item for $n=9$: $\left\{  \operatorname{AGL}\left(  1,9\right)
,\operatorname{A \Gamma L}\left(  1,9\right)  \right\}  $, $\left\{
\operatorname{ASL}\left(  2,3\right)  ,\operatorname{AGL}\left(  2,3\right)
\right\}  $\newline\hspace*{12ex}and $\left\{  \operatorname{PSL}\left(
2,8\right)  ,\operatorname{P \Gamma L}\left(  2,8\right)  ,A_{9}%
,S_{9}\right\}  ;$

\item for $n=10$: $\left\{  \operatorname{PGL}\left(  2,9\right)
,\operatorname{P \Gamma L}\left(  2,9\right)  \right\}  $ and $\left\{
A_{10},S_{10}\right\}  .$
\end{enumerate}
\end{theorem}

In our terminology, Theorem~\ref{thm seress} states that for $n\geq11$ every
primitive subgroup of $S_{n}$ except $A_{n}$ is Galois closed over
$\mathbf{2}$, whereas for $n\leq10$ the only primitive subgroups of $S_{n}$
that are not Galois closed over $\mathbf{2}$ are the ones listed above
(omitting the last group from each block, which is the closure of the other
groups in the same block).



\begin{center}
\vspace{1.5ex} \textbf{(C) Direct and subdirect products }
\end{center}

\nopagebreak In the sequel, $B$ and $D$ always denote disjoint subsets of
$\mathbf{n}$ such that $\mathbf{n}=B\cup D$, and $G\times H$ stands for the
direct product of $G\leq S_{B}$ and $H\leq S_{D}$. In this paper we only
consider direct products with the intransitive action, i.e., the two groups
act independently on disjoint sets. Given permutations $\beta\in S_{B}$ and
$\delta\in S_{D}$, we write $\beta\times\delta$ for the corresponding element
of $S_{B}\times S_{D}$. Let $\pi_{1}$ and $\pi_{2}$ denote the first and
second projections on the direct product $S_{B}\times S_{D}$. Then we have
$\pi_{1}\left(  \beta\times\delta\right)  =\beta$ and $\pi_{2}\left(
\beta\times\delta\right)  =\delta$ for every $\beta\in S_{B},\delta\in S_{D}$,
and $\sigma=\pi_{1}\left(  \sigma\right)  \times\pi_{2}\left(  \sigma\right)
$ for every $\sigma\in S_{B}\times S_{D}$.

Recall that a subdirect product is a subgroup of a direct product such that
the projection to each coordinate is surjective. Hence, if $G\leq S_{B}\times
S_{D}$ and $G_{1}=\pi_{1}\left(  G\right)  $, $G_{2}=\pi_{2}\left(  G\right)
$, then $G$ is a subdirect product of $G_{1}$ and $G_{2}$. We denote this fact
by $G\leq_{\operatorname{sd}}G_{1}\times G_{2}$, and by $G<_{\operatorname{sd}%
}G_{1}\times G_{2}$ we mean a proper subdirect subgroup of $G_{1}\times G_{2}%
$. According to Remak~\cite{remak}, the following description of subdirect
products of groups is due to Klein~\cite{klein}. (Of course, the theorem is
valid for abstract groups, not just for permutation groups. For an English
reference, see Theorem~5.5.1 of \cite{hall}.)

\begin{theorem}
[\cite{klein,remak}]\label{thm subdirect}If $G\leq_{\operatorname{sd}}%
G_{1}\times G_{2}$, then there exists a group $K$ and surjective homomorphisms
$\varphi_{i}\colon G_{i}\rightarrow K\left(  i=1,2\right)  $ such that%
\[
G=\left\{  g_{1}\times g_{2}\mid\varphi_{1}\left(  g_{1}\right)  =\varphi
_{2}\left(  g_{2}\right)  \right\}  .
\]

\end{theorem}

Note that in the above theorem we have $G=G_{1}\times G_{2}$ if and only if
$K$ is the trivial (one-element) group.

\section{The main result and some general observations\label{section general}}

Our main result is the following partial solution of
Problem~\ref{problem main} for the case $n\gg d=n-k$.

\begin{theorem}
\label{thm main}Let $n>\max\left(  2^{d},d^{2}+d\right)  $ and $G\leq S_{n}$.
Then $G$ is not Galois closed over $\mathbf{k}$ if and only if $G=A_{B}\times
L$ or $G<_{\operatorname{sd}}S_{B}\times L$, where $B\subseteq\mathbf{n}$ is
such that $D:=\mathbf{n}\setminus B$ has less than $d$ elements, and $L$ is an
arbitrary permutation group on $D$.
\end{theorem}


Note that the set $D$ in the theorem above is much smaller than $B$, thus $B$
is a \textquotedblleft big\textquotedblright\ subset of $\mathbf{n}$, and
$L\leq S_{D}$ is a \textquotedblleft little group\textquotedblright, hence the
notation. The subdirect product $G<_{\operatorname{sd}}S_{B}\times L$ is not
determined by $B$ and $L$, but in Proposition~\ref{prop subdirect A_B and S_B}
we give a fairly concrete description of these groups.
Proposition~\ref{prop closure of subdirect product} shows that the groups
given in Theorem~\ref{thm main} are indeed not Galois closed over $\mathbf{k}$
(and that their Galois closure is $S_{B}\times L$). In
Section~\ref{section proof of thm main} we will prove that these are the only
non-closed groups; however, already in this section we present the proof for
the case $d=1$ (i.e.,\ $k=n-1$), which illustrates the main ideas of the proof
of the general case.


\begin{center}
\vspace{1.5ex} \textbf{(A) The case $k=n-1$ }
\end{center}

From (\ref{eq closure via orbits}) we can derive the following useful formula
for the Galois closure of a group, which has been discovered independently by
K.~Kearnes \cite{Kearnes}. Here $\left(  S_{n}\right)  _{a}$ denotes the
stabilizer of $a\in\mathbf{k}^{n}$ under the action of $S_{n}$. Note that this
stabilizer is the direct product of symmetric groups on the sets $\left\{
i\in\mathbf{n}\mid a_{i}=j\right\}  $, $j\in\mathbf{k}$.

\begin{proposition}
\label{prop kearnes}For every $G\leq S_{n}$, we have%
\[
\overline{G}^{\left(  k\right)  }=\bigcap_{a\in\mathbf{k}^{n}}\left(
S_{n}\right)  _{a}\cdot G.
\]

\end{proposition}

\begin{proof}
We reformulate the condition $a^{\sigma}\in a^{G}$ of
(\ref{eq closure via orbits}) for $a\in\mathbf{k}^{n},\sigma\in S_{n}$ as
follows:%
\begin{align*}
a^{\sigma}\in a^{G}  &  \iff\exists\pi\in G:~a^{\sigma}=a^{\pi}\\
&  \iff\exists\pi\in G:a^{\sigma\pi^{-1}}=a\\
&  \iff\exists\pi\in G:\sigma\pi^{-1}\in\left(  S_{n}\right)  _{a}\\
&  \iff\sigma\in\left(  S_{n}\right)  _{a}\cdot G.
\end{align*}
Now from (\ref{eq closure via orbits}) it follows that $\sigma\in\overline
{G}^{\left(  k\right)  }$ if and only if $\sigma\in\left(  S_{n}\right)
_{a}\cdot G$ holds for all $a\in\mathbf{k}^{n}$.
\end{proof}

With the help of Proposition~\ref{prop kearnes}, we can prove that all
subgroups of $S_{n}$ are Galois closed over $\mathbf{k}$ if and only if $k\geq
n$.

\begin{proposition}
\label{prop3.3} \label{prop k>=n} If $k\geq n\geq2$, then each subgroup $G\leq
S_{n}$ is Galois closed over $\mathbf{k}$; if $2\leq k<n$, then $A_{n}$ is not
Galois closed over $\mathbf{k}$.
\end{proposition}

\begin{proof}
Clearly, if $k\geq n$ then there exists a tuple $a\in\mathbf{k}^{n}$ whose
components are pairwise different. Consequently, $(S_{n})_{a}$ is trivial and
therefore $\overline{G}^{\left(  k\right)  }\subseteq(S_{n})_{a}\cdot G=G$ for
all $G\leq S_{n}$ by Proposition~\ref{prop kearnes}. On the other hand, if
$k<n$ then there is a repetition in every tuple $a\in\mathbf{k}^{n}$, hence
$\left(  S_{n}\right)  _{a}$ contains a transposition. Therefore $\left(
S_{n}\right)  _{a}\cdot A_{n}=S_{n}$ for all $a\in\mathbf{k}^{n}$, thus
$\overline{A_{n}}^{\left(  k\right)  }=S_{n}$ by
Proposition~\ref{prop kearnes}.
\end{proof}

Now we can solve Problem~\ref{problem main} in the case $k=n-1$, which is the
simplest nontrivial case. The proof of the following proposition already
contains the key steps of the proof of Theorem~\ref{thm main}.

\begin{proposition}
\label{prop k=n-1}For $k=n-1\geq2$, each subgroup of $S_{n}$ except $A_{n}$ is
Galois closed over $\mathbf{k}$.
\end{proposition}

\begin{proof}
If $G\leq S_{n}$ is not Galois closed over $\mathbf{k}$, then
Proposition~\ref{prop kearnes} shows that for all $\pi\in\overline{G}^{\left(
k\right)  }\setminus G$ and for all $a\in\mathbf{k}^{n}$, we have $\pi
\in\left(  S_{n}\right)  _{a}\cdot G$, hence $\pi=\gamma\sigma$ for some
$\gamma\in\left(  S_{n}\right)  _{a}$ and $\sigma\in G$. Therefore,
$\gamma=\pi\sigma^{-1}\in\overline{G}^{\left(  k\right)  }$; moreover,
$\gamma\neq\operatorname*{id}$ follows from $\pi\notin G$. Thus we see that
$\overline{G}^{\left(  k\right)  }$ contains at least one non-identity
permutation from every stabilizer:%
\begin{equation}
\overline{G}^{\left(  k\right)  }\neq G\implies\forall a\in\mathbf{k}%
^{n}~\exists\gamma\in\left(  S_{n}\right)  _{a}\setminus\left\{
\operatorname*{id}\right\}  :~\gamma\in\overline{G}^{\left(  k\right)  }.
\label{eq key idea}%
\end{equation}

Now fix $i,j\in\mathbf{n}$, $i\neq j$, and let $a=(a_{1},\dots,a_{n}%
)\in\mathbf{k}^{n}$ be a tuple such that $a_{r}=a_{s}\iff\{r,s\}=\{i,j\}\text{
or }r=s$. Then $\left(  S_{n}\right)  _{a}=\left\{  \operatorname*{id},\left(
ij\right)  \right\}  $, where $\left(  ij\right)  \in S_{n}$ denotes the
transposition of $i$ and $j$. Applying (\ref{eq key idea}), we see that
$\left(  ij\right)  \in\overline{G}^{\left(  k\right)  }$ for all
$i,j\in\mathbf{n}$, hence $\overline{G}^{\left(  k\right)  }=S_{n}$. From
Proposition~\ref{prop kearnes} it follows that $\overline{G}^{\left(
k\right)  }\subseteq\left(  S_{n}\right)  _{a}\cdot G\subseteq S_{n}%
=\overline{G}^{\left(  k\right)  }$, i.e., $S_{n}=\left(  S_{n}\right)
_{a}\cdot G$ for every $a\in\mathbf{k}^{n}$. Choosing $a$ as above, we have
$S_{n}=\left\{  \operatorname*{id},\left(  ij\right)  \right\}  \cdot G$,
hence $G$ is of index at most $2$ in $S_{n}$. Therefore, we have either
$G=A_{n}$ or $G=S_{n}$; the latter is obviously Galois closed, whereas $A_{n}$
is not Galois closed over $\mathbf{k}$ by Proposition~\ref{prop k>=n}.
\end{proof}

Clote and Kranakis \cite{CloK91} define a group $G\leq S_{n}$ to be
\emph{weakly representable}, if there exist positive integers $k,m$ with
$2\leq k<n$ and $2\leq m$ such that $G$ is the invariance group of some
function $f\colon\mathbf{k}^{n}\rightarrow\mathbf{m}$ (equivalently, $G$ is
$\left(  k,\infty\right)  $-representable for some $k<n$).
Proposition~\ref{prop k>=n} shows that the restriction $k<n$ is important;
allowing $k=n$ would make all groups weakly representable.
Proposition~\ref{prop k=n-1} yields a complete description of weakly
representable groups.

\begin{corollary}
All subgroups of $G\leq S_{n}$ except for $A_{n}$ are weakly representable.
\end{corollary}

\begin{proof}
According to Fact~\ref{fact closed vs repr}, a subgroup of $S_{n}$ is weakly
representable if and only if it is Galois closed over $\mathbf{k}$ for some
$k<n$. This is equivalent to being Galois closed over $\mathbf{n-1}$, as the
closures for $k=2,3,\ldots,n-1$ form a descending chain (see
(\ref{eq chain of closures}) in Section~\ref{section computations}). From
Proposition~\ref{prop k=n-1} it follows that all subgroups of $S_{n}$ are
Galois closed over $\mathbf{n-1}$ except for $A_{n}$.
\end{proof}

\begin{center}

\vspace{1.5ex} \textbf{(B) Closures of direct and subdirect products }
\end{center}

The following proposition describes closures of direct products, and, as a
corollary, we obtain a generalization of \cite[Theorem 3.1]{Kis1998}.

\begin{proposition}
\label{prop closure of direct product}For all $G\leq S_{B}$ and $H\leq S_{D}$,
we have $\overline{G\times H}^{\left(  k\right)  }=\overline{G}^{\left(
k\right)  }\times\overline{H}^{\left(  k\right)  }$.
\end{proposition}

\begin{proof}
For notational convenience, let us assume that $B=\left\{  1,\ldots,t\right\}
$ and $D=\left\{  t+1,\ldots,n\right\}  $. If $a=\left(  1,\ldots
,1,2,\ldots,2\right)  \in\mathbf{k}^{n}$ with $t$ ones followed by $n-t$ twos,
then the stabilizer of $a$ in $S_{n}$ is $S_{B}\times S_{D}$. Hence from
Proposition~\ref{prop kearnes} it follows that $\overline{G\times H}^{\left(
k\right)  }\leq\left(  S_{B}\times S_{D}\right)  \cdot\left(  G\times
H\right)  =S_{B}\times S_{D}$, i.e., every element of $\overline{G\times
H}^{\left(  k\right)  }$ is of the form $\beta\times\delta$ for some $\beta\in
S_{B},\delta\in S_{D}$. For arbitrary $a=\left(  a_{1},\ldots,a_{n}\right)
\in\mathbf{k}^{n}$, let $a_{B}=\left(  a_{1},\ldots,a_{t}\right)
\in\mathbf{k}^{t}$ and $a_{D}=\left(  a_{t+1},\ldots,a_{n}\right)
\in\mathbf{k}^{n-t}$. It is straightforward to verify that $a^{\beta
\times\delta}\in a^{G\times H}$ if and only if $a_{B}^{\beta}\in a_{B}^{G}$
and $a_{D}^{\delta}\in a_{D}^{H}$. Thus applying (\ref{eq closure via orbits}%
), we have%
\begin{align*}
\beta\times\delta\in\overline{G\times H}^{\left(  k\right)  }  &  \iff\forall
a\in\mathbf{k}^{n}:a^{\beta\times\delta}\in a^{G\times H}\\
&  \iff\forall a\in\mathbf{k}^{n}:\left(  a_{B}^{\beta}\in a_{B}^{G}\text{ and
}a_{D}^{\delta}\in a_{D}^{H}\right) \\
&  \iff\left(  \forall a_{B}\in\mathbf{k}^{t}:a_{B}^{\beta}\in a_{B}%
^{G}\right)  \text{ and }\left(  \forall a_{D}\in\mathbf{k}^{n-t}%
:a_{D}^{\delta}\in a_{D}^{H}\right) \\
&  \iff\beta\in\overline{G}^{\left(  k\right)  }\text{ and }\delta\in
\overline{H}^{\left(  k\right)  }\\
&  \iff\beta\times\delta\in\overline{G}^{\left(  k\right)  }\times\overline
{H}^{\left(  k\right)  }.%
\qedhere
\end{align*}

\end{proof}

\begin{corollary}
\label{cor GxH closed iff G and H closed}For all $G\leq S_{B}$ and $H\leq
S_{D}$, the direct product $G\times H$ is Galois closed over $\mathbf{k}$ if
and only if both $G$ and $H$ are Galois closed over $\mathbf{k}$.
\end{corollary}

\begin{proof}
The \textquotedblleft if\textquotedblright\ part follows immediately from
Proposition~\ref{prop closure of direct product}. For the \textquotedblleft
only if\textquotedblright\ part, assume that $G\times H$ is Galois closed over
$\mathbf{k}$. From Proposition~\ref{prop closure of direct product} we get
$G\times H=\overline{G}^{\left(  k\right)  }\times\overline{H}^{\left(
k\right)  }$, and this implies $G=\overline{G}^{\left(  k\right)  }$ and
$H=\overline{H}^{\left(  k\right)  }$.
\end{proof}

\begin{remark}
\label{remark fixed point}If $n<m$, then any subgroup $G$ of $S_{n}$ can be
naturally embedded into $S_{m}$ as the subgroup $G\times\left\{
\operatorname{id}_{\mathbf{m}\setminus\mathbf{n}}\right\}  $. From
Proposition~\ref{prop closure of direct product} it follows that
$\overline{G\times\left\{  \operatorname{id}_{\mathbf{m}\setminus\mathbf{n}%
}\right\}  }^{\left(  k\right)  }=\overline{G}^{\left(  k\right)  }%
\times\left\{  \operatorname{id}_{\mathbf{m}\setminus\mathbf{n}}\right\}  $,
i.e., there is no danger of ambiguity in not specifying whether we regard $G$
as a subgroup of $S_{n}$ or as a subgroup of $S_{m}$.
\end{remark}

\begin{remark}
Proposition~\ref{prop closure of direct product} and
Corollary~\ref{cor GxH closed iff G and H closed} do not generalize to
subdirect products. It is possible that a subdirect product of two Galois
closed groups is not Galois closed. For example, let
\[
G=\left\{  \operatorname*{id},\left(  123\right)  ,\left(  132\right)
,\left(  12\right)  \left(  45\right)  ,\left(  13\right)  \left(  45\right)
,\left(  23\right)  \left(  45\right)  \right\}  <_{\operatorname{sd}%
}S_{\left\{  1,2,3\right\}  }\times S_{\left\{  4,5\right\}  };
\]
then $\overline{G}^{\left(  2\right)  }=S_{\left\{  1,2,3\right\}  }\times
S_{\left\{  4,5\right\}  }$, hence $G$ is not Galois closed over $\mathbf{2}$.
It is also possible that a subdirect product is closed, although the factors
are not both closed: let%
\[
G=\left\{  \operatorname*{id},\left(  13\right)  \left(  24\right)  ,\left(
1234\right)  \left(  56\right)  ,\left(  1432\right)  \left(  56\right)
\right\}  <_{\operatorname{sd}}\langle\left(  1234\right)  \rangle
\times\langle(56)\rangle;
\]
then $G$ is Galois closed over $\mathbf{2}$, but the $4$-element cyclic group
is not Galois closed over $\mathbf{2}$ (its Galois closure is the dihedral
group of degree $4$).


\end{remark}

Next we determine (the closures of) the special subdirect products involving
symmetric and alternating groups that appear in Theorem~\ref{thm main}.

\begin{proposition}
\label{prop subdirect A_B and S_B}Let $\left\vert B\right\vert >\max\left(
\left\vert D\right\vert ,4\right)  $ and $L\leq S_{D}$. If $G\leq
_{\operatorname{sd}}A_{B}\times L$, then $G=A_{B}\times L$. If $G\leq
_{\operatorname{sd}}S_{B}\times L$, then either $G=S_{B}\times L$, or there
exists a subgroup $L_{0}\leq L$ of index $2$, such that%
\begin{equation}
G=\left(  A_{B}\times L_{0}\right)  \cup%
\bigl(%
\left(  S_{B}\setminus A_{B}\right)  \times\left(  L\setminus L_{0}\right)
\bigr)%
. \label{eq subdirect}%
\end{equation}

\end{proposition}

\begin{proof}
Suppose that $G\leq_{\operatorname{sd}}A_{B}\times L$, and let $K$ and
$\varphi_{1},\varphi_{2}$ be as in Theorem~\ref{thm subdirect} (for
$G_{1}=A_{B}$ and $G_{2}=L$). Since $A_{B}$ is simple, the kernel of
$\varphi_{1}$ is either $\left\{  \operatorname{id}_{B}\right\}  $ or $A_{B}$.
In the first case, $K$ is isomorphic to $A_{B}$; however, this cannot be a
homomorphic image of $L$, as $| L|\leq|S_{D}|<|A_{B}|$.
In the second case, $K$ is trivial and $G=A_{B}\times L$. If $G\leq
_{\operatorname{sd}}S_{B}\times L$, then there are three possibilities for the
kernel of $\varphi_{1}$, namely $\left\{  \operatorname{id}_{B}\right\}  $,
$A_{B}$ and $S_{B}$. Just as above, the first case is impossible, while in the
third case we have $G=S_{B}\times L$. In the second case, $K$ is a two-element
group, hence by letting $L_{0}$ be the kernel of $\varphi_{2}$, we obtain
(\ref{eq subdirect}).
\end{proof}



\begin{proposition}
\label{prop closure of subdirect product} Let $|D|<d\leq n-d$ and let $G$ be
any one of the subdirect products considered in
Proposition~\ref{prop subdirect A_B and S_B}. Then $\overline{G}^{\left(
k\right)  }=S_{B}\times L$.
\end{proposition}

\begin{proof}
Since $k=n-d>\left\vert D\right\vert $, all subgroups of $S_{D}$ are closed by
Proposition~\ref{prop k>=n}, hence $\overline{L}^{\left(  k\right)  }=L$. On
the other hand, $k<\left\vert B\right\vert $ implies that $A_{B}$ is not
closed; in fact, we have $\overline{A_{B}}^{\left(  k\right)  }=S_{B}$.
Therefore $\overline{A_{B}\times L}^{\left(  k\right)  }=\overline{A_{B}%
}^{\left(  k\right)  }\times\overline{L}^{\left(  k\right)  }=S_{B}\times L$,
and also $\overline{S_{B}\times L}^{\left(  k\right)  }=S_{B}\times L$. It
remains to consider the case when $G$ is of the form~(\ref{eq subdirect}).
Then we have $A_{B}\times L_{0}\leq G\leq S_{B}\times L$, thus%
\begin{equation}
S_{B}\times L_{0}=\overline{A_{B}\times L_{0}}^{\left(  k\right)  }%
~{\leq_{\ast}}~\overline{G}^{\left(  k\right)  }\leq\overline{S_{B}\times
L}^{\left(  k\right)  }=S_{B}\times L. \label{eq closure of subdirect product}%
\end{equation}
Moreover, $\overline{G}^{\left(  k\right)  }$ contains $\left(  S_{B}\setminus
A_{B}\right)  \times\left(  L\setminus L_{0}\right)  $, and this shows that
the first containment in (\ref{eq closure of subdirect product}) (marked with
asterisk) is strict. However, $S_{B}\times L_{0}$ is of index $2$ in
$S_{B}\times L$, therefore we can conclude that $\overline{G}^{\left(
k\right)  }=S_{B}\times L$.
\end{proof}

\section{Proof of Theorem~\ref{thm main}\label{section proof of thm main}}

The proof of Theorem~\ref{thm main} is based on the same idea as that of
Proposition~\ref{prop k=n-1}:%
\renewcommand{\theenumi}{\arabic{enumi})}
\renewcommand{\labelenumi}{\theenumi}%

\begin{enumerate}
\item first we use (\ref{eq key idea}) with specific tuples $a$ to show that
$\overline{G}^{\left(  k\right)  }$ must be a \textquotedblleft
large\textquotedblright\ group
(see Subsection~\ref{section proof of thm main}(A) below), and then

\item we prove that $G$ is of \textquotedblleft small\textquotedblright\ index
in $\overline{G}^{\left(  k\right)  }$
(see Subsection~\ref{section proof of thm main}(B) below).
\end{enumerate}

For the first step, we will need to apply (\ref{eq key idea}) for several
groups acting on different sets, hence, for easier reference, we give a name
to this property.

\begin{definition}
Let $\Omega\subseteq\mathbf{n}$ be a nonempty set, and let us consider the
natural action of $S_{\Omega}$ on $\mathbf{k}^{\Omega}$ for a positive integer
$k\geq2$. We say that $H\leq S_{\Omega}$ is $k$\emph{-thick}, if%
\[
\forall a\in\mathbf{k}^{\Omega}:~\exists\gamma\in\left(  S_{\Omega}\right)
_{a}\setminus\left\{  {\operatorname*{id}}_{\Omega}\right\}  :~\gamma\in H.
\]

\end{definition}

We will use thickness with two types of tuples $a\in\mathbf{k}^{\Omega}$.
First, let $a$ contain only one repeated value, which is repeated exactly
$d+1$ times, say at the coordinates $i_{1},\ldots,i_{d+1}\in\Omega$ (note that
such a tuple exists only if $\left\vert \Omega\right\vert \geq d+1$). Then the
stabilizer of $a$ is the full symmetric group on $\left\{  i_{1}%
,\ldots,i_{d+1}\right\}  $, therefore $k$-thickness of $H$ implies that%
\begin{equation}
\exists\gamma\in S_{\left\{  i_{1},\ldots,i_{d+1}\right\}  }\setminus\left\{
\operatorname{id}\right\}  :~\gamma\in H. \label{eq o}%
\end{equation}
Next, let $d$ values be repeated in $a$, each of them repeated exactly two
times, say at the coordinates $i_{1},j_{1};~i_{2},j_{2};\ldots;i_{d},j_{d}$
(here we need $\left\vert \Omega\right\vert \geq2d$). Then the stabilizer of
$a$ is the group generated by the transpositions $\left(  i_{1}j_{1}\right)
,\left(  i_{2}j_{2}\right)  ,\ldots,\left(  i_{d}j_{d}\right)  $. Thus
$k$-thickness of $H$ implies that
\begin{equation}
\exists\gamma\in\langle\left(  i_{1}j_{1}\right)  ,\left(  i_{2}j_{2}\right)
,\ldots,\left(  i_{d}j_{d}\right)  \rangle\setminus\left\{  \operatorname{id}%
\right\}  :~\gamma\in H. \label{eq oo}%
\end{equation}

The first paragraph of the proof of Proposition~\ref{prop k=n-1} can be
reformulated as follows:

\begin{fact}
\label{fact closure of nonclosed is thick}If $G\leq S_{n}$ is not Galois
closed over $\mathbf{k}$, then $\overline{G}^{\left(  k\right)  }$ is $k$-thick.
\end{fact}



\begin{center}

\vspace{1.5ex} \textbf{(A) The closures of non-closed groups }
\end{center}

The goal of this subsection is to prove the following description of the
closures of non-closed groups.

\begin{proposition}
\label{prop closure of nonclosed}Let $n>d^{2}+d$. If $G\leq S_{n}$ is not
Galois closed over $\mathbf{k}$, then $\overline{G}^{\left(  k\right)  }$ is
of the form $S_{B}\times L$, where $B\subseteq\mathbf{n}$ is such that
$D:=\mathbf{n}\setminus B$ has less than $d$ elements, and $L$ is a
permutation group on $D$.
\end{proposition}

Throughout this subsection we will always assume that $G<\overline{G}^{\left(
k\right)  }\leq S_{n}$ with $n>d^{2}+d$, where $d=n-k\geq1$. We consider the
action of $\overline{G}^{\left(  k\right)  }$ on $\mathbf{n}$ (not on
$\mathbf{k}^{n}$), and we separate two cases upon the transitivity of this
action. First we deal with the transitive case, for which we will make use of
the following theorem of Bochert \cite{Bochert89} (see also \cite{DixMo,Wie}).

\begin{theorem}
[\cite{Bochert89}]\label{thm bochert}If $G$ is a primitive subgroup of
$S_{\Omega}$ not containing $A_{\Omega}$, then there exists a subset
$I\subseteq\Omega$ with $\left\vert I\right\vert \leq\frac{\left\vert
\Omega\right\vert }{2}$ such that the pointwise stabilizer of $I$ in $G$ is trivial.
\end{theorem}

\begin{lemma}
\label{lemma transitive H}Let $\Omega\subseteq\mathbf{n}$ such that
$\left\vert \Omega\right\vert >\max\left(  2d,d^{2}\right)  $. If $H$ is a
transitive $k$-thick subgroup of $S_{\Omega}$, then $H=A_{\Omega}$ or
$H=S_{\Omega}$.
\end{lemma}

\begin{proof}
Assume for contradiction that $H$ satisfies the assumptions of the lemma, but
$H$ does not contain $A_{\Omega}$. If $H$ is primitive, then let us consider
the set $I$ given in Theorem~\ref{thm bochert}. Since $\left\vert
\Omega\setminus I\right\vert \geq\frac{\left\vert \Omega\right\vert }{2}>d$,
we can find $d+1$ elements $i_{1},\ldots,i_{d+1}$ in $\Omega\setminus I$.
Since $H$ is $k$-thick and $\left\vert \Omega\right\vert \geq d+1$, we can
apply (\ref{eq o}) for $i_{1},\ldots,i_{d+1}$, and we obtain a permutation
$\gamma\neq\operatorname{id}$ in the pointwise stabilizer of $I$ in $H$, which
is a contradiction.

Thus $H$ cannot be primitive. Since it is transitive, there exists a
nontrivial partition
\begin{equation}
\Omega=B_{1}\dot{\cup}\cdots\dot{\cup}B_{r} \label{eq partition}%
\end{equation}
with $\left\vert B_{1}\right\vert =\cdots=\left\vert B_{r}\right\vert =s$ and
$r,s\geq2$ such that every element of $H$ preserves this partition. We will
prove by contradiction that $r\leq d$ and $s\leq d$. First let us assume that
$r>d$; let $B_{1}=\left\{  i_{1},j_{1},\ldots\right\}  ,\ldots,B_{d+1}%
=\left\{  i_{d+1},j_{d+1},\ldots\right\}  $, and let $\gamma$ be the
permutation provided by (\ref{eq o}). Since $\gamma\neq\operatorname{id}$,
there exist $p,q\in\left\{  1,\ldots,d+1\right\}  ,p\neq q$ such that
$\gamma\left(  i_{p}\right)  =i_{q}$. On the other hand, we have
$\gamma\left(  j_{p}\right)  =j_{p}$, and this means that $\gamma$ does not
preserve the partition (\ref{eq partition}). Next let us assume that $s>d$;
let $B_{1}=\left\{  i_{1},\ldots,i_{d+1},\ldots\right\}  ,~B_{2}=\left\{
j_{1},\ldots,j_{d+1},\ldots\right\}  $, and let $\gamma$ be the permutation
provided by (\ref{eq oo}). Since $\gamma\neq\operatorname{id}$, there exists
$p\in\left\{  1,\ldots,d\right\}  $ such that $\gamma\left(  i_{p}\right)
=j_{p}$. On the other hand, we have $\gamma\left(  i_{d+1}\right)  =i_{d+1}$,
and this means that $\gamma$ does not preserve the partition
(\ref{eq partition}). We can conclude that $r,s\leq d$, hence we have
$\left\vert \Omega\right\vert =rs\leq d^{2}<\left\vert \Omega\right\vert $, a contradiction.
\end{proof}

\begin{lemma}
\label{lemma transitive case}If $\overline{G}^{\left(  k\right)  }$ is
transitive, then $\overline{G}^{\left(  k\right)  }=S_{n}$.
\end{lemma}

\begin{proof}
Since $n>d^{2}+d$, we have $n>\max\left(  2d,d^{2}\right)  $. Thus from
Fact~\ref{fact closure of nonclosed is thick} and
Lemma~\ref{lemma transitive H} it follows that either $\overline{G}^{\left(
k\right)  }=A_{n}$ or $\overline{G}^{\left(  k\right)  }=S_{n}$. However,
$A_{n}$ is not Galois closed over $\mathbf{k}$ by Proposition~\ref{prop k>=n},
because $n>k$.
\end{proof}

Now let us consider the intransitive case. The first step is to prove that in
this case there is a unique \textquotedblleft big\textquotedblright\ orbit.

\begin{lemma}
\label{lemma big orbit}If $\overline{G}^{\left(  k\right)  }$ is not
transitive, then it has an orbit $B$ such that $D=\mathbf{n}\setminus B$ has
less than $d$ elements.
\end{lemma}

\begin{proof}
We claim that $\overline{G}^{\left(  k\right)  }$ has at most $d$ orbits.
Suppose to the contrary, that there exists $d+1$ elements $i_{1}%
,\ldots,i_{d+1}\in\mathbf{n}$, each belonging to a different orbit. If
$\gamma\in\overline{G}^{\left(  k\right)  }$ is the permutation given by
(\ref{eq o}), then there exist $p,q\in\left\{  1,\ldots,d+1\right\}  $, $p\neq
q$ such that $\gamma\left(  i_{p}\right)  =i_{q}$, and this contradicts the
fact that $i_{p}$ and $i_{q}$ belong to different orbits of $\overline
{G}^{\left(  k\right)  }$. Now, the average orbit size is at least $\frac
{n}{d}>d$, therefore there exists an orbit $B=\left\{  i_{1},\ldots
,i_{d},\ldots\right\}  $ of size at least $d$. We will show that the
complement of $B$ has at most $d-1$ elements. Suppose this is not true, i.e.,
there are at least $d$ elements $j_{1},\ldots,j_{d}$ outside $B$. With the
help of (\ref{eq oo}) we obtain a permutation $\gamma\in\overline{G}^{\left(
k\right)  }$ for which there exists $p\in\left\{  1,\ldots,d\right\}  $ such
that $\gamma\left(  i_{p}\right)  =j_{p}$. This is clearly a contradiction,
since $i_{p}$ belongs to the orbit $B$, whereas $j_{p}$ belongs to some other orbit.
\end{proof}

At this point we know that $\overline{G}^{\left(  k\right)  }\leq S_{B}\times
S_{D}$. Using the the notation $G_{1}=\pi_{1}%
\bigl(%
\,\overline{G}^{\left(  k\right)  }%
\bigr)%
$ and $L=\pi_{2}%
\bigl(%
\,\overline{G}^{\left(  k\right)  }%
\bigr)%
$ for the projections of $\overline{G}^{\left(  k\right)  }$, we have
$\overline{G}^{\left(  k\right)  }\leq_{\operatorname{sd}}G_{1}\times L$.

\begin{lemma}
\label{lemma intransitive case}If $\overline{G}^{\left(  k\right)  }$ is not
transitive and $B$ is the big orbit given in Lemma~\ref{lemma big orbit}, then
$\overline{G}^{\left(  k\right)  }=S_{B}\times L$ for some $L\leq S_{D}$.
\end{lemma}

\begin{proof}
First we show that $G_{1}$ inherits $k$-thickness from $\overline{G}^{\left(
k\right)  }$. Let $b\in\mathbf{k}^{B}$, and extend $b$ to a tuple
$a\in\mathbf{k}^{n}$ such that the components $a_{i}\left(  i\in D\right)  $
are pairwise different (this is possible, since $\left\vert D\right\vert <k$).
The $k$-thickness of $\overline{G}^{\left(  k\right)  }$ implies that there
exists a permutation $\gamma\in\left(  S_{n}\right)  _{a}\cap\overline
{G}^{\left(  k\right)  }\setminus\left\{  \operatorname*{id}\right\}  $, and
from $\overline{G}^{\left(  k\right)  }\leq_{\operatorname{sd}}G_{1}\times L$
it follows that $\gamma=\beta\times\delta$ for some $\beta\in G_{1},\delta\in
L$. The construction of the tuple $a$ ensures that $\delta=\operatorname*{id}%
_{D}$, hence we have $\operatorname*{id}_{B}\neq\beta\in\left(  S_{B}\right)
_{b}\cap G_{1}$, and this proves that $G_{1}$ is a $k$-thick subgroup of
$S_{B}$.

Since $B$ is an orbit of $\overline{G}^{\left(  k\right)  }$, the action of
$G_{1}$ on $B$ is transitive. From $n>d^{2}+d$ it follows that $\left\vert
B\right\vert =n-\left\vert D\right\vert >n-d\geq\max\left(  2d,d^{2}\right)
$, hence Lemma~\ref{lemma transitive H} shows that $G_{1}\geq A_{B}$. This
means that either $\overline{G}^{\left(  k\right)  }\leq_{\operatorname{sd}%
}A_{B}\times L$ or $\overline{G}^{\left(  k\right)  }\leq_{\operatorname{sd}%
}S_{B}\times L$. Now with the help of
Proposition~\ref{prop subdirect A_B and S_B} and
Proposition~\ref{prop closure of subdirect product} we can conclude that
$\overline{G}^{\left(  k\right)  }=S_{B}\times L$. (Note that the assumption
$\left\vert B\right\vert >4$ in Proposition~\ref{prop subdirect A_B and S_B}
is not satisfied if $d=1$ and $n\leq4$. However, $d=1$ implies $D=\emptyset$,
what contradicts the intransitivity of $\overline{G}^{\left(  k\right)  }$.)
\end{proof}

Combining Lemmas~\ref{lemma transitive case} and \ref{lemma intransitive case}%
, we obtain Proposition~\ref{prop closure of nonclosed}, q.e.d.


\begin{center}
\vspace{1.5ex} \textbf{(B) The non-closed groups }
\end{center}

In this subsection we prove the following Proposition~\ref{prop nonclosed}. It
describes the groups $G$ with $\overline{G}^{\left(  k\right)  }=S_{B} \times
L$ and therefore completes also the proof of Theorem~\ref{thm main}.

\begin{proposition}
\label{prop nonclosed}Let $n>\max\left(  2^{d},d^{2}+d\right)  $, let
$B\subseteq\mathbf{n}$ and $D=\mathbf{n}\setminus B$ such that $\left\vert
D\right\vert <d$, and let $L\leq S_{D}$. If $G\leq S_{n}$ is a group whose
Galois closure over $\mathbf{k}$ is $S_{B}\times L$, then $G\leq
_{\operatorname{sd}}A_{B}\times L$ or $G\leq_{\operatorname{sd}}S_{B}\times L$.
\end{proposition}

Throughout this subsection we will assume that $n>\max\left(  2^{d}%
,d^{2}+d\right)  $, where $d=n-k\geq1$, and $\overline{G}^{\left(  k\right)
}=S_{B}\times L$, where $B$ and $L$ are as in the proposition above. Let
$G_{1}=\pi_{1}\left(  G\right)  \leq S_{B}$ and $G_{2}=\pi_{2}\left(
G\right)  \leq S_{D}$; then we have $G\leq_{\operatorname{sd}}G_{1}\times
G_{2}$. As in
Subsection~\ref{section proof of thm main}(A), we begin with the transitive
case (i.e., $D=\emptyset$), and we will use the following well-known result
(see, e.g., \cite[Exercise 14.3]{Wie}).


\begin{proposition}
\label{thm sz-cz-sz}If $n>4$ and $H$ is a proper subgroup of $S_{n}$ different
from $A_{n}$, then the index of $H$ is at least $n$.
\end{proposition}

\begin{lemma}
\label{lemma closure is S_n}If $\overline{G}^{\left(  k\right)  }=S_{n}$, then
$G=A_{n}$ or $G=S_{n}$.
\end{lemma}

\begin{proof}
Let $a\in\mathbf{k}^{n}$ be the tuple which was used to obtain (\ref{eq oo});
then we have $\left(  S_{n}\right)  _{a}=\langle\left(  i_{1}j_{1}\right)
,\left(  i_{2}j_{2}\right)  ,\ldots,\left(  i_{d}j_{d}\right)  \rangle$. From
Proposition~\ref{prop kearnes} we obtain%
\[
S_{n}=\overline{G}^{\left(  k\right)  }\subseteq\left(  S_{n}\right)
_{a}\cdot G,
\]
hence we have $\left(  S_{n}\right)  _{a}\cdot G=S_{n}$. Since $\left\vert
\left(  S_{n}\right)  _{a}\right\vert =2^{d}$, the index of $G$ in $S_{n}$ is
at most $2^{d}<n$, and therefore Proposition~\ref{thm sz-cz-sz} implies that
$G\geq A_{n}$ if $n>4$. If $n\leq4$, then $d=1$, thus we can apply
Proposition~\ref{prop k=n-1}.
\end{proof}

\begin{lemma}
\label{lemma closure is S_B x Delta}If $\overline{G}^{\left(  k\right)
}=S_{n}\times L$, then $G_{1}\geq A_{n}$ and $G_{2}= L$.
\end{lemma}

\begin{proof}
Clearly, $G\leq G_{1}\times G_{2}$ implies $S_{B}\times L=\overline
{G}^{\left(  k\right)  }\leq\overline{G_{1}\times G_{2}}^{\left(  k\right)
}=\overline{G_{1}}^{\left(  k\right)  }\times\overline{G_{2}}^{\left(
k\right)  }$ by Proposition~\ref{prop closure of direct product}. It follows
that $\overline{G}_{1}^{\left(  k\right)  }=S_{B}$, and thus
Lemma~\ref{lemma closure is S_n} yields $G_{1}\geq A_{n}$. (One can verify
that the inequality $n>\max\left(  2^{d},d^{2}+d\right)  $ holds if we write
$\left\vert B\right\vert $ in place of $n$ and $\left\vert B\right\vert -k$ in
place of $d$.) On the other hand, $k>\left\vert D\right\vert $ implies that
$\overline{G_{2}}^{\left(  k\right)  }=G_{2}$ by Proposition~\ref{prop k>=n},
hence%
\[
G\leq S_{B}\times L=\overline{G}^{\left(  k\right)  }\leq\overline{G_{1}%
}^{\left(  k\right)  }\times\overline{G_{2}}^{\left(  k\right)  }%
=\overline{G_{1}}^{\left(  k\right)  }\times G_{2}.
\]
Applying $\pi_{2}$ to these inequalities, we obtain $G_{2}\leq L\leq G_{2}$,
and this proves $G_{2}=L$.
\end{proof}

Since $G\leq_{\operatorname{sd}}G_{1}\times G_{2}$,
Lemma~\ref{lemma closure is S_B x Delta} immediately implies
Proposition~\ref{prop nonclosed}, q.e.d.

\section{Computational results\label{section computations}}

The Galois closures of a group $G\leq S_{n}$ over $\mathbf{k}$ for
$k=2,3,\ldots$ form a nonincreasing sequence, eventually stabilizing at $G$
itself:%
\begin{equation}
\overline{G}^{\left(  2\right)  }\geq\overline{G}^{\left(  3\right)  }%
\geq\cdots\geq\overline{G}^{\left(  n-1\right)  }\geq\overline{G}^{\left(
n\right)  }=\overline{G}^{\left(  n+1\right)  }=\cdots=G.
\label{eq chain of closures}%
\end{equation}
We computed the Galois closures of all subgroups of $S_{n}$ for $2\leq k\leq
n\leq6$ by computer, and we found that for most of these groups the chain of
closures contains only $G$ (i.e., $G$ is Galois closed over $\mathbf{2}$), and
for all other groups (\ref{eq chain of closures}) consists only of two
different groups (namely $\overline{G}^{\left(  2\right)  }$ and $G$).
Table~\ref{table n<=6} shows the list of groups corresponding to the latter
case, up to conjugacy. For each group, the first column gives the smallest $n$
for which $G$ can be embedded into $S_{n}$ (here we mean an embedding as a
permutation group, not as an abstract group; cf.
Remark~\ref{remark fixed point}). We also give the largest $k$ such that
$\overline{G}^{\left(  k\right)  }\neq G$, i.e., (\ref{eq chain of closures})
takes the form $\overline{G}^{\left(  2\right)  }=\ldots=\overline{G}^{\left(
k\right)  }>\overline{G}^{\left(  k+1\right)  }=\ldots=G$.%

\begin{table}[tbp] \centering
\begin{tabular}
[c]{lll}\hline
\multicolumn{1}{|l}{} & \multicolumn{1}{|l}{$G\leq S_{n}$} &
\multicolumn{1}{|l|}{$\overline{G}^{\left(  k\right)  }$}\\\hline\hline
\multicolumn{1}{|l}{$n=3,\,k=2\,$} & \multicolumn{1}{|l}{$A_{3}$} &
\multicolumn{1}{|l|}{$S_{3}$}\\\hline
\multicolumn{1}{|l}{$n=4,\,k=3$} & \multicolumn{1}{|l}{$A_{4}$} &
\multicolumn{1}{|l|}{$S_{4}$}\\\hline
\multicolumn{1}{|l}{$n=4,\,k=2$} & \multicolumn{1}{|l}{$C_{4}$} &
\multicolumn{1}{|l|}{$D_{4}$}\\\hline
\multicolumn{1}{|l}{$n=5,\,k=4$} & \multicolumn{1}{|l}{$A_{5}$} &
\multicolumn{1}{|l|}{$S_{5}$}\\\hline
\multicolumn{1}{|l}{$n=5,\,k=2$} & \multicolumn{1}{|l}{$\operatorname*{AGL}%
\left(  1,5\right)  $} & \multicolumn{1}{|l|}{$S_{5}$}\\\hline
\multicolumn{1}{|l}{$n=5,\,k=2$} & \multicolumn{1}{|l}{$S_{3}\times
_{\operatorname{sd}}S_{2}$} & \multicolumn{1}{|l|}{$S_{3}\times S_{2}$%
}\\\hline
\multicolumn{1}{|l}{$n=5,\,k=2$} & \multicolumn{1}{|l}{$A_{3}\times S_{2}$} &
\multicolumn{1}{|l|}{$S_{3}\times S_{2}$}\\\hline
\multicolumn{1}{|l}{$n=5,\,k=2$} & \multicolumn{1}{|l}{$C_{5}$} &
\multicolumn{1}{|l|}{$D_{5}$}\\\hline
\multicolumn{1}{|l}{$n=6,\,k=5$} & \multicolumn{1}{|l}{$A_{6}$} &
\multicolumn{1}{|l|}{$S_{6}$}\\\hline
\multicolumn{1}{|l}{$n=6,\,k=2$} & \multicolumn{1}{|l}{$\operatorname*{PGL}%
\left(  2,5\right)  $} & \multicolumn{1}{|l|}{$S_{6}$}\\\hline
\multicolumn{1}{|l}{$n=6,\,k=3$} & \multicolumn{1}{|l}{$S_{4}\times
_{\operatorname{sd}}S_{2}$} & \multicolumn{1}{|l|}{$S_{4}\times S_{2}$%
}\\\hline
\multicolumn{1}{|l}{$n=6,\,k=3$} & \multicolumn{1}{|l}{$A_{4}\times S_{2}$} &
\multicolumn{1}{|l|}{$S_{4}\times S_{2}$}\\\hline
\multicolumn{1}{|l}{$n=6,\,k=2$} & \multicolumn{1}{|l}{$S_{3}\times
_{\operatorname{sd}}S_{3}$} & \multicolumn{1}{|l|}{$S_{3}\times S_{3}$%
}\\\hline
\multicolumn{1}{|l}{$n=6,\,k=2$} & \multicolumn{1}{|l}{$A_{3}\times S_{3}$} &
\multicolumn{1}{|l|}{$S_{3}\times S_{3}$}\\\hline
\multicolumn{1}{|l}{$n=6,\,k=2$} & \multicolumn{1}{|l}{$D_{4}\times
_{\operatorname{sd}}S_{2}$} & \multicolumn{1}{|l|}{$D_{4}\times S_{2}$%
}\\\hline
\multicolumn{1}{|l}{$n=6,\,k=2$} & \multicolumn{1}{|l}{$C_{4}\times S_{2}$} &
\multicolumn{1}{|l|}{$D_{4}\times S_{2}$}\\\hline
\multicolumn{1}{|l}{$n=6,\,k=3$} & \multicolumn{1}{|l}{$\left(  S_{3}\wr
S_{2}\right)  \cap A_{6}$} & \multicolumn{1}{|l|}{$S_{3}\wr S_{2}$}\\\hline
\multicolumn{1}{|l}{$n=6,\,k=2$} & \multicolumn{1}{|l}{$S_{3}\wr
_{\operatorname{sd}}S_{2}$} & \multicolumn{1}{|l|}{$S_{3}\wr S_{2}$}\\\hline
\multicolumn{1}{|l}{$n=6,\,k=2$} & \multicolumn{1}{|l}{$R\left(
\mbox{\mancube}\right)  $} & \multicolumn{1}{|l|}{$S\left(
\mbox{\mancube}\right)  $}\\\hline
&  &
\end{tabular}
\caption{Nontrivial closures for $n \le 6$}\label{table n<=6}%
\end{table}%

Some of the entries in Table~\ref{table n<=6} may need some explanation. Using
the notation of Theorem~\ref{thm subdirect}, each subdirect product in the
table corresponds to a two-element quotient group $K$: for symmetric groups
$S_{n}$ we take the homomorphism $\varphi\colon S_{n}\rightarrow K$ with
kernel $A_{n}$ (cf. Proposition~\ref{prop subdirect A_B and S_B}), whereas for
the dihedral group $D_{4}$ we take the homomorphism $\varphi\colon
D_{4}\rightarrow K$ whose kernel is the group of rotations in $D_{4}$. The
group $S_{3}\wr S_{2}$ is the wreath product of $S_{3}$ and $S_{2}$ (with the
imprimitive action); equivalently, it is the semidirect product $\left(
S_{3}\times S_{3}\right)  \rtimes S_{2}$ (with $S_{2}$ acting on the direct
product by permuting the two components). By $S_{3}\wr_{\operatorname{sd}%
}S_{2}$ we mean the \textquotedblleft subdirect wreath
product\textquotedblright\ $\left(  S_{3}\times_{\operatorname{sd}}%
S_{3}\right)  \rtimes S_{2}$. Finally, the groups $S\left(
\mbox{\mancube}\right)  $ and $R\left(  \mbox{\mancube}\right)  $ denote the
group of all symmetries and the group of all rotations (orientation-preserving
symmetries) of the cube, acting on the six faces of the cube.

Combining these computational results with Theorem~\ref{thm main}, we get the
solution of Problem~\ref{problem main} for the case $d=2$.

\begin{proposition}
\label{prop k=n-2}For $k=n-2\geq2$, each subgroup of $S_{n}$ except $A_{n}$
and $A_{n-1}$ (for $n\geq4$) and $C_{4}$ (for $n=4$) is Galois closed over
$\mathbf{k}$.
\end{proposition}

\begin{proof}
If $n>6$, then we can apply Theorem~\ref{thm main}, and we obtain the
exceptional groups $A_{n}$ and $A_{n-1}$ from the direct product $A_{B}\times
L$ with $\left\vert D\right\vert =0$ and $\left\vert D\right\vert =1$,
respectively. If $n\leq6$, then the non-closed groups can be read from
Table~\ref{table n<=6}.
\end{proof}

We have also examined the linear groups appearing in Theorem~\ref{thm seress}
by computer, and we have found that all of them are Galois closed over
$\mathbf{3}$. Thus we have the following result for primitive groups.

\begin{proposition}
\label{prop primitive 3-closed}Every primitive permutation group except for
$A_{n}\left(  n\geq4\right)  $ is Galois closed over $\mathbf{3}$.
\end{proposition}

\section{Concluding remarks and open problems\label{section concluding}}

We have introduced a Galois connection to study invariance groups of
$n$-variable functions defined on a $k$-element domain, and we have studied
the corresponding closure operator. Our main result is that if the difference
$d=n-k$ is relatively small compared to $n$, then \textquotedblleft most
groups\textquotedblright\ are Galois closed, and we have explicitly described
the non-closed groups. The bound $\max\left(  2^{d},d^{2}+d\right)  $ of
Theorem~\ref{thm main} is probably not the best possible; it remains an open
problem to improve it.

\begin{problem}
Determine the smallest number $f\left(  d\right)  $ such that
Theorem~\ref{thm main} is valid for all $n\geq f\left(  d\right)  $.
\end{problem}

For fixed $d$, the inequality $n>\max\left(  2^{d},d^{2}+d\right)  $ fails
only for \textquotedblleft small\textquotedblright\ values of $n$, so one
might hope that these cases can be dealt with easily. However, our
investigations indicate that there is a simple pattern in the closures if $n$
is much larger than $d$, and exactly those exceptional\ groups corresponding
to small values of $n$ are the ones that make the problem difficult. (We can
say that the Boolean case is the hardest, as in this case $n$ is just $d+2$.)
We have fully settled only the cases $d\leq2$; perhaps it is feasible to
attack the problem for the next few values of $d$.

\begin{problem}
Describe the (non-)closed groups for $d=3,4,\ldots$.
\end{problem}

The chain of closures (\ref{eq chain of closures}) for the groups that we
investigated in our computer experiments has length at most two: for all
$k\geq2$, we have either $\overline{G}^{\left(  k\right)  }=\overline
{G}^{\left(  2\right)  }$ or $\overline{G}^{\left(  k\right)  }=G$. This is
certainly not true in general; for example, we have
\[
\overline{A_{3}\times\cdots\times A_{t}}^{\left(  k\right)  }=A_{3}%
\times\cdots\times A_{k}\times S_{k+1}\times\cdots\times S_{t},
\]
hence $\overline{G}^{\left(  2\right)  }>\overline{G}^{\left(  3\right)
}>\cdots>\overline{G}^{\left(  t-1\right)  }>\overline{G}^{\left(  t\right)
}=G$ holds for $G=A_{3}\times\cdots\times A_{t}$. It is natural to ask if
there exist groups with long chains of closures that are not direct products
of groups acting on smaller sets. As Proposition~\ref{prop primitive 3-closed}
shows, we cannot find such groups among primitive groups.

\begin{problem}
\label{problem transitive long chain of closures}Find transitive groups with
arbitrarily long chains of closures.
\end{problem}

The closure operator defined in Section~\ref{section preliminaries}(A)
concerns the Galois closure with respect to the Galois connection induced by
the relation $\vdash\,\subseteq S_{n}\times O_{k}^{(n)}$, based on a natural
action of $S_{n}$ on $\mathbf{k}^{n}$ (see Section~\ref{section introduction}%
). In permutation group theory also another closure operator, called
$k$\emph{-closure} is used, which was introduced by H.~Wielandt
(\cite[Definition~5.3]{Wi69}). This notion describes the Galois closures with
respect to the Galois connection induced by the relation $\triangleright
\subseteq S_{n}\times{\mathfrak{P}}(\mathbf{n}^{k})$. Here $\sigma\in S_{n}$
acts on $r=(r_{1},\dots,r_{k})\in\mathbf{n}^{k}$ according to $r^{\sigma
}:=(r_{1}\sigma,\dots,r_{k}\sigma)$, and, for a $k$-ary relation
$\varrho\subseteq\mathbf{n}^{k}$, we have $\sigma\triangleright\varrho$ if and
only if $\sigma$ preserves $\varrho$, i.e., $r^{\sigma}\in\varrho$ for all
$r\in\varrho$. For $G\subseteq S_{n}$, the $k$-closure $(G^{\triangleright
})^{\triangleright}$ is denoted by $\operatorname*{Aut}\operatorname{Inv}%
^{(k)}G$ (\cite{PoeK79}), or by $G^{(k)}=\operatorname{gp}(\mbox{$k$-rel }G)$
(\cite{Wi69}). A group $G\leq S_{n}$ is $k$-closed if and only if it can be
defined by $k$-ary relations, i.e., if there exists a set $R$ of $k$-ary
relations on $\mathbf{n}$ such that $G$ consists of the permutations that
preserve every member of $R$. The following proposition establishes a
connection between the two notions of closure.

\begin{proposition}
For every $G\leq S_{n}$ and $k\geq1$, the Galois closure $\overline
{G}^{\left(  k+1\right)  }$ is contained in the $k$-closure of $G$. In
particular, every $k$-closed group is Galois closed over $\mathbf{k+1}$.
\end{proposition}

\begin{proof}
The proof is based on a suitable correspondence between $\mathbf{n}^{k}$ and
$\left(  \mathbf{k+1}\right)  ^{n}$. Let $r=(r_{1},\dots,r_{k})\in
\mathbf{n}^{k}$ be a $k$-tuple whose components are pairwise different. We
define $\varkappa\left(  r\right)  =\left(  a_{1},\ldots,a_{n}\right)
\in\left(  \mathbf{k+1}\right)  ^{n}$ as follows:%
\[
a_{i}=\left\{  \!\!%
\begin{array}
[c]{rl}%
\ell, & \text{if }i=r_{\ell};\\
k+1, & \text{if }i\notin\left\{  r_{1},\ldots,r_{k}\right\}  .
\end{array}
\right.
\]
Thus $\varkappa$ is a partial map from $\mathbf{n}^{k}$ to $\left(
\mathbf{k+1}\right)  ^{n}$, and it is straightforward to verify that
$\varkappa$ is injective, and $\varkappa\left(  r\right)  ^{\sigma^{-1}%
}=\varkappa\left(  r^{\sigma}\right)  $ holds for all $\sigma\in S_{n}$ and
$r\in\mathbf{n}^{k}$ with mutually different components. (Here $\varkappa
\left(  r\right)  ^{\sigma^{-1}}$ refers to the action of $S_{n}$ on $\left(
\mathbf{k+1}\right)  ^{n}$ by permuting the components of $n$-tuples, while
$r^{\sigma}$ refers to the action of $S_{n}$ on $\mathbf{n}^{k}$ by mapping
$k$-tuples componentwise.)

Now let $G\leq S_{n}$ and $\pi\in\overline{G}^{\left(  k+1\right)  }$; we need
to show that $r^{\pi}\in r^{G}$ for every $r\in\mathbf{n}^{k}$. We may assume
that the components of $r$ are pairwise distinct (otherwise we can remove the
repetitions and work with a smaller $k$). From $\pi\in\overline{G}^{\left(
k+1\right)  }$ it follows that $\varkappa\left(  r\right)  ^{\pi^{-1}}%
\in\varkappa\left(  r\right)  ^{G}$. Therefore, we have $\varkappa\left(
r^{\pi}\right)  =\varkappa\left(  r\right)  ^{\pi^{-1}}\in\varkappa\left(
r\right)  ^{G}=\varkappa\left(  r^{G}\right)  $, and then the injectivity of
$\varkappa$ gives that $r^{\pi}\in r^{G}$.
\end{proof}

Note that the proposition above implies that each group that is not Galois
closed over $\mathbf{k}$ (such as the ones in Theorem~\ref{thm main}) is also
an example of a permutation group that cannot be characterized by $(k-1)$-ary relations.

The connection between the two notions of closure in the other direction is
much weaker. For example, the \emph{Mathieu group} $M_{12}$ is Galois closed
over $\mathbf{2}$ (since it is the automorphism group of a hypergraph), but it
is not $5$-closed (since it is $5$-transitive, and this implies that the
$5$-closure of $M_{12}$ is the full symmetric group $S_{12}$). In some sense,
this is a worst possible case, as it is not difficult to prove that if a
subgroup of $S_{n}$ is Galois closed over $\mathbf{2}$, then it is
$\left\lfloor \frac{n}{2}\right\rfloor $-closed (in particular, $M_{12}$ is
$6$-closed).

\begin{problem}
Determine the smallest number $w\left(  n,k\right)  $ such that every subgroup
of $S_{n}$ that is Galois closed over $\mathbf{k}$ is also $w\left(
n,k\right)  $-closed.
\end{problem}

\subsection*{Acknowledgement}

The authors are grateful to Keith Kearnes, Erkko Lehtonen and S\'{a}ndor
Ra\-de\-lecz\-ki for stimulating discussions, and also to P\'{e}ter P\'{a}l
P\'{a}lfy who suggested the example mentioned before
Problem~\ref{problem transitive long chain of closures}.

\end{document}